\documentclass[11pt]{article}

\usepackage[margin=1in]{geometry}

\usepackage{amsmath,amsthm,amssymb,amsfonts}
\usepackage{graphicx}
\usepackage[T1]{fontenc}

\DeclareMathOperator{\mex}{mex}

\newcommand{\imark}{i\textsc{-Mark}}

\newcommand{\N}{{\mathbb N}}
\newcommand{\R}{{\mathbb R}}

\newtheorem{theorem}{Theorem}
\newtheorem{lemma}[theorem]{Lemma}
\newtheorem{observation}[theorem]{Observation}

\theoremstyle{definition}
\newtheorem{definition}[theorem]{Definition}

\title{A convergence technique for the game $\imark$}
\author{Gabriel Nivasch\thanks{Corresponding author. Ariel University, Ariel, Israel. \texttt{gabrieln@ariel.ac.il}.}\and Oz Rubinstein\thanks{Work was done while the author was at Ariel University. \texttt{rubinsteinoz5@gmail.com}.}}
\date{}

\begin{document}
	
\maketitle

\begin{abstract}
	The game of $\imark$ is an impartial combinatorial game introduced by Sopena (2016). The game is parametrized by two sets of positive integers $S$, $D$, where $\min D\ge 2$. From position $n\ge 0$ one can move to any position $n-s$, $s\in S$, as long as $n-s\ge 0$, as well as to any position $n/d$, $d\in D$, as long as $n>0$ and $d$ divides $n$. The game ends when no more moves are possible, and the last player to move is the winner. Sopena, and subsequently Friman and Nivasch (2021), characterized the Sprague--Grundy sequences of many cases of $\imark(S,D)$ with $|D|=1$. Friman and Nivasch also obtained some partial results for the case $\imark(\{1\},\{2,3\})$.
	
	In this paper we present a convergence technique that gives polynomial-time algorithms for the Sprague--Grundy sequence of many instances of $\imark$ with $|D|>1$. In particular, we prove our technique works for all games $\imark(\{1\},\{d_1,d_2\})$.
	
\end{abstract}

\section{Introduction}

The game $\imark$ is an impartial combinatorial game introduced by Sopena \cite{SOPENA201690}. The game is parametrized by two sets of positive integers $S$ and $D$, where $\min D\ge 2$. A position in the game is given by a single natural number $n$. There are two players, and in each turn, a player can move from position $n$ to position $n-s$ for for some $s\in S$ (assuming $n-s\ge 0$), or to position $n/d$ for some $d\in D$, if $n/d$ is a positive integer. The game ends when no more moves are possible.

As usual in the theory of impartial games, in \emph{normal play} the last player to move is the winner (or in other words, the player with no legal moves is the loser).

For example, in the game $\imark(\{1\},\{2,3\})$, starting from position $10$, the first player can move to position $9$ or $5$. If she moves to position $9$, then the second player can move to position $8$ or $3$, and so on. The player who moves to position $0$ will be the winner.

\subsection{The Sprague--Grundy theory of impartial games}

An \emph{impartial game} is a two-player game in which there is perfect information, there is no chance involved, and, from each position, the legal moves available are the same for both players (see e.g.~Conway \cite{conway2000numbers}, Berlekamp et.~al~\cite{WINNING}). An impartial game can be represented by a directed acyclic graph $G=(V,E)$, where each position in the game is represented by a vertex, and each legal move from one position to another is represented by an edge. We assume that the game is \emph{well-founded}, i.e.~from no position in the game there exists an infinite sequence of moves.

We say that position $v$ is a \emph{follower} of position $u$ if one can get from position $u$ to position $v$ in one move.

In an impartial game, positions are classified into \emph{$P$-positions} (in which, with perfect play, the previous player to have moved is the winner), and \emph{$N$-positions} (in which, with perfect play, the next player to move is the winner). $P$- and $N$-positions are characterized recursively as follows: If all the followers of position $v$ are $N$-positions, then $v$ is a $P$-position. Otherwise (if $v$ has at least one follower which is a $P$-position), then $v$ is an $N$-position. The base cases of this characterization are the \emph{sinks}, or positions with no followers, which are $P$-positions.

Let $G_1=(V_1,E_1)$ and $G_2=(V_2,E_2)$ be two impartial games. Their \emph{sum} $G'=G_1+G_2$ is the impartial game played as follows: Each position of $G'$ consists of a pair of positions $v_1v_2$ with $v_1\in V_1$ and $v_2\in V_2$. On each turn, a player chooses one of the games $G_1$, $G_2$ and plays on it, leaving the other game untouched. The game ends when no moves are possible on $G_1$ nor on $G_2$. Formally, $G'=(V',E')$ where
\begin{align*}
	V'&=V_1\times V_2,\\
	E'&=\{(v_1v_2, w_1v_2):(v_1,w_1)\in E_1\}\\&\qquad\cup\{(v_1v_2,v_1w_2):(v_2,w_2)\in E_2\}.
\end{align*}

In order to know whether a position $v'=v_1v_2$ in the game $G'$ is a $P$- or an $N$-position, it is not enough to know whether $v_1$ and $v_2$ are $P$- or $N$-positions: It can be shown that by induction that, if both $v_1$ and $v_2$ are $P$-positions, then $v'$ is a $P$-position, while if exactly one of $v_1, v_2$ is a $P$-position, then $v'$ is an $N$-position. However, if both $v_1, v_2$ are $N$-positions, then $v'$ could be either a $P$- or an $N$-position.

The \emph{Sprague--Grundy function} (or \emph{Grundy function} for short) enables us to correctly play sums of impartial games (Sprague~\cite{SPRAGUE}, Grundy~\cite{GRUNDY}). Denote $\N=\{0,1,2,\ldots\}$. Let $G$ be an impartial game. To each position $v$ of $G$ let us assign a \emph{Grundy value} $\mathcal G(v)\in\N$ recursively as follows: $\mathcal G(v) = \mex{\{\mathcal G(w): \text{$w$ is a follower of $v$}\}}$. Here $\mex S = \min(\N\setminus S)$ denotes the \emph{minimum excluded value}. Then, the following two properties can be proven by induction:
\begin{itemize}
	\item We have $\mathcal G(v)=0$ if and only if $v$ is a $P$-position.
	\item If $v'=v_1v_2$ is a position in the game $G'=G_1+G_2$, then $\mathcal G(v')=\mathcal G(v_1)\oplus\mathcal G(v_2)$, where $\oplus$ is the \emph{nim-sum}, or bitwise-XOR, operation (so for example $3\oplus 5=011_2\oplus 101_2=110_2=6$).
\end{itemize}

Hence, by knowing the Grundy function of an impartial game we can correctly play this game, as well as the sum of this game with any other impartial game we also know the Grundy function of.

Since the ``sum of games'' operation is commutative and associative, as is the nim-sum operation, the Sprague--Grundy function enables us to play sums of an arbitrary number of games.

\subsection{\boldmath Previous work on $\imark$}

Let $\imark(S,D)$ be a fixed $\imark$ game. Suppose we are given an integer $n$, and we want to find the value of $\mathcal G(n)$. The naive way of doing this is to compute the entire sequence $\mathcal G(0), \mathcal G(1), \ldots, \mathcal G(n)$ from the bottom up. This approach takes time and space $\Theta(n)$. However, since the input size is $\log_2 n$, this naive algorithm takes time and space that is exponential in the size of the input.

To \emph{solve} a particular instance $\imark(S,D)$ means to find a polynomial-time algorithm that, given a position $n$, returns $\mathcal G(n)$. This can be done by identifying a pattern in the Grundy sequence of the game (for example, all $\imark$ games with $D=\emptyset$ are eventually periodic), or by some other method.

Sopena \cite{SOPENA201690} examined and solved many cases of $\imark$ with $|D|=1$. Subsequently, Friman and Nivasch \cite{FN} solved some additional cases, also with $|D|=1$, that were left as open problems in \cite{SOPENA201690}. In addition, \cite{FN} derived some partial results for the case $\imark(\{1\},\{2,3\})$ (which is arguably the simplest case with $|D|>1$): They derived upper bounds for the maximum gaps between consecutive occurrences of the values $0$, $1$, and $2$ in the Grundy sequence. They observed experimentally that there also seems to be an upper bound on the gap between consecutive occurrences of Grundy value $3$, though they did not manage to prove that.

\subsection{Our results}

In this paper we present a \emph{convergence}-based technique that yields polynomial-time algorithms for many instances of $\imark$. In particular, we prove that our convergence technique works in all instances of the form $\imark(\{1\},\{d_1, d_2\})$. Hence, all games $\imark(\{1\},\{d_1, d_2\})$ have polynomial-time algorithms. We also present experimental results regarding convergence in other instances of $\imark$.

Our convergence technique for $\imark$ is somewhat similar to the convergence approach used for the game of Wythoff in \cite{nivasch2005more}.

\section{The convergence technique}

If $s\in S$ then we say that position $n-s$ is a \emph{subtraction follower} of $n$. If $d\in D$ and $d|n$ then we say that $n/d$ is a \emph{division follower} of $n$.

In general in impartial games, if a position has $k$ followers, then its Grundy value must be a number between $0$ and $k$. In the case of $\imark(S,D)$, if $n\ge \max S$, then the number of followers of $n$ is $|S|$ plus the number of elements of $D$ that divide $n$. Let $\varphi(n)$ denote the number of followers of $n$. Hence, $0\le \mathcal G(n)\le \varphi(n)$.

The idea behind our technique is the following ``convergence'' phenomenon. Suppose we want to compute $\mathcal G$ on some interval $J=[x,y]=\{x,x+1,\ldots,y\}$. In order to do this, we first extend our interval downwards by an appropriate constant amount $c$ (which will hopefully depend only on the parameters $S$ and $D$ of the game), getting the interval $I=[x-c,y]$. Now suppose that for every $d\in D$ we have already computed $\mathcal G$ on the interval $I_d=\{\lfloor(x-c)/d\rfloor, \cdots, \lfloor y/d\rfloor\}$. In order to compute $\mathcal G$ on the interval $I$ itself, we would still need to know the values of $\mathcal G$ on the $s=\max S$ first elements of $I$. Once these Grundy values are given, we could continue the computation of $\mathcal G$ on $I$ from the bottom up.

However, if we do not know the values of $\mathcal G$ on the first $s$ elements of $I$, there is still a possibility that the following will work: We try all possible assignments of Grundy values to these $s$ positions. For each assignment, we continue the computation of $\mathcal G$ upwards. If, after some number of steps, all computations happen to agree with each other on $s$ consecutive values, then they will keep agreeing with each other from that point on. In this case, we say that ``convergence'' has occurred. If convergence occurs after $c$ or fewer steps, then we will obtain the desired Grundy values on our interval $J$. See Figure \ref{fig_convergence} for an illustration of convergence with $S=\{1\}$.

Let us formulate the convergence phenomenon formally. Let a game $\imark(S,D)$ be fixed. Let $s=\max S$. Let $n$ be a starting position. A \emph{guess seed} for starting position $n$ is a sequence $\overline \sigma = (\sigma_n, \ldots, \sigma_{n+s-1})$ where $0\le \sigma_i\le\varphi(i)$ for each $n\le i<n+s$. The \emph{guess sequence} $\mathcal G_{\overline\sigma}$ corresponding to the guess seed $\sigma$ is the infinite sequence $\mathcal G_{\overline\sigma}=(\mathcal G_{\overline\sigma}(n), \mathcal G_{\overline\sigma}(n+1),\ldots)$, where $\mathcal G_{\overline\sigma}(i)=\sigma_i$ for $n\le i<n+s$, and $\mathcal G_{\overline\sigma}(m) = \mex \{\mathcal G_{\overline\sigma}(m-s_i):s_i\in S\}\cup \{\mathcal G(m/d_i):d_i\in D, d_i|m\}$ for all $m\ge n+s$. Meaning, the sequence $\mathcal G_{\overline\sigma}$ is computed from the bottom up by starting from the seed $\overline\sigma$ as presumed Grundy values, using previously computed guesses for the subtraction followers, and using the actual Grundy values for the division followers.

Let $\Sigma_n$ be the set of all possible guess seeds for starting position $n$. Note that $|\Sigma_n|=\prod_{i=n}^{n+s-1}(1+\varphi(i))$. One of these seeds is the correct one, i.e.~the one that satisfies $\sigma_i=\mathcal G(i)$ for all $n\le i<n+s$.

We say that \emph{convergence occurs} in $c$ steps starting at position $n$ if $c\ge s$ and $\mathcal G_{\overline\sigma}(m)=\mathcal G_{\overline\sigma'}(m)$ for all pairs of seeds $\overline\sigma, \overline\sigma'\in\Sigma_n$ and all $n+c\le m<n+c+s$. Note that in this case, we will have $\mathcal G_{\overline\sigma}(m)=\mathcal G_{\overline\sigma'}(m)$ for all $m\ge n+c+s$ as well. Since one of the seeds in $\Sigma_n$ is the correct one, convergence implies that the values $\mathcal G_{\overline\sigma}(m)$ equal $\mathcal G(m)$ for all $m\ge n+c$.

\begin{figure}
	\centering
	\includegraphics[width=16cm]{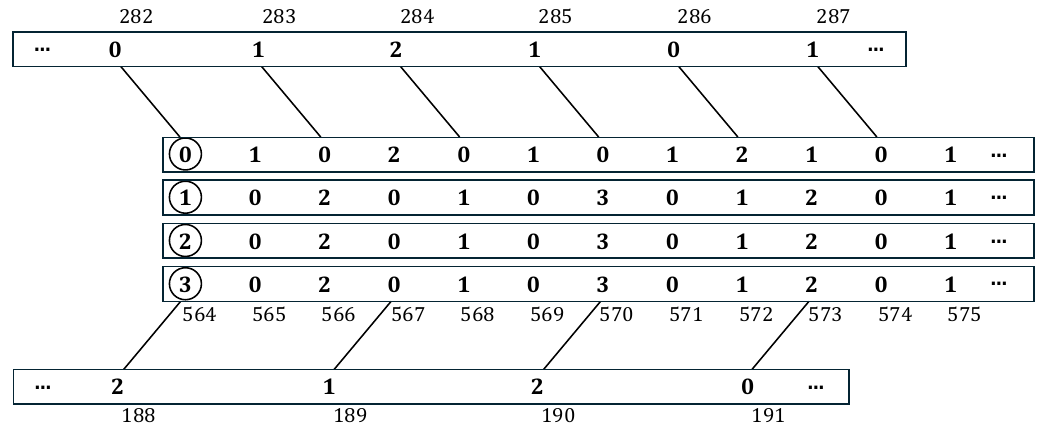}
	\caption{\label{fig_convergence}In the game $\imark(\{1\},\{2,3\})$, convergence starting from position $564$ occurs after $10$ steps. The four sequences at the center are the guess sequences, and the circled number in each guess sequence is its seed.}
\end{figure}

We say that for the game $\imark(S,D)$ convergence occurs in $c$ steps, if for every starting position $n$ convergence occurs in at most $c$ steps.

\subsection{The polynomial-time algorithm}

The convergence phenomenon, when it occurs, yields a polynomial-time algorithm. The general idea behind the algorithm is as follows: In order to compute $\mathcal G(n)$, we try to compute $\mathcal G$ on an interval $I$ that ends at $n$, using the above-mentioned convergence approach. Since convergence takes some steps to occur, only a suffix of the interval $I$ will contain definite Grundy values. In order to perform the computation on $I$, we recursively compute $\mathcal G$ on lower intervals. In general, some of the lower intervals will be used multiple times. Hence, a dynamic programming approach will prevent us from computing the same values multiple times, so the running time will be kept polynomial in $\log n$.

Let us now describe the algorithm in detail.

The following technical lemma will be helpful:

\begin{lemma}\label{lem_floors}
	Let $x\in\R^+$ be a positive real number, and let $n_1,\ldots,n_k\in\N^+$ be positive integers. Then
	\begin{equation}\label{eq_floors}
		\lfloor\lfloor\cdots\lfloor \lfloor \lfloor x\rfloor / n_1\rfloor /n_2\rfloor /\cdots \rfloor / n_k\rfloor = \lfloor x/(n_1 n_2\cdots n_k)\rfloor.
	\end{equation}
\end{lemma}
	
\begin{proof}
	We first prove the claim for the case $k=1$. Suppose for a contradiction that $\lfloor \lfloor x\rfloor/n_1\rfloor < \lfloor x/n_1\rfloor$. Then there exists an integer $m$ that satisfies $\lfloor x\rfloor / n_1 < m\le x/n_1$. But this implies $\lfloor x\rfloor < mn_1 \le x$, implying that $mn_1$ is not an integer. Contradiction.
	
	Hence, in the left-hand side of (\ref{eq_floors}) we can remove one by one all the floor operators except for the outermost one, obtaining the right-hand side of (\ref{eq_floors}).
\end{proof}

The algorithm is as follows: Suppose that on $\imark(S,D)$, $D=\{d_1, \ldots, d_k\}$, convergence occurs in at most $c$ steps. Let $n$ be the given a position on which we want to compute $\mathcal G$. Define the set $M\subseteq\N$ by
\begin{equation*}
	M=\{\lfloor n/(d_1^{a_1}\cdots d_k^{a_k})\rfloor : a_1, \ldots, a_k\in\N\}.
\end{equation*}
Note that $n\in M$ (by taking $a_1=\cdots=a_k=0$). And note that $|M|=O(\log^{k} n)$, since taking $a_i> \log_2 n$ for some $i$ gives a result of $0$.

For each $m\in M$ define the intervals $I_m=[m-2c,m]$ and $J_m=[m-c,m]$. In order to compute $\mathcal G(n)$, we will compute $\mathcal G$ on all intervals $J_m\cap\N$, $m\in M$, by increasing order of $m$.

This is done as follows: Let $m\in M$, and suppose we have already computed $J_{m'}\cap\N$ for all $m'\in M$, $m'<m$.

If $m-2c\le 0$ then we just compute $\mathcal G$ on $J_m\cap \N$ from the bottom up in the naive way.

Now suppose $m-2c> 0$. By Lemma \ref{lem_floors}, for each $d_i\in D$ we have $\lfloor m/d_i\rfloor\in M$. Furthermore, for each $x\in I_m$ that is divisible by $d_i$, we have $x/d_i\in J_{\lfloor m/d_i\rfloor}$. Meaning, we have the actual Grundy values of all the division followers of the positions in $I_m$. We compute all possible guess sequences $\mathcal G_{\overline\sigma}$ starting from the beginning of $I_m$, until we reach the end of $I_m$. By assumption, convergence will occur in at most $c$ steps, so all the guess sequences will agree with the real Grundy values within the interval $J_m$. See Figure \ref{fig_alg}.

\begin{figure}
	\centering
	\includegraphics{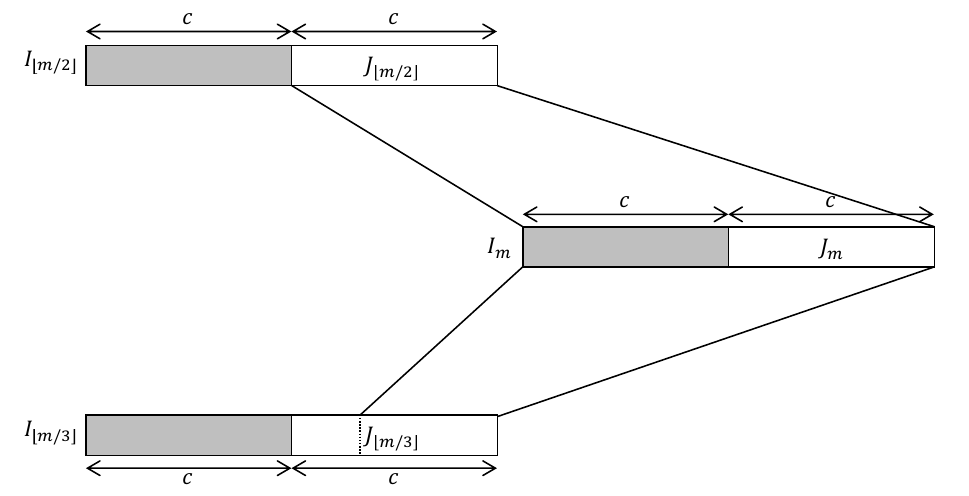}
	\caption{\label{fig_alg}Computation of guess sequences on an interval $I_m$ based on the Grundy values of previously computed intervals. In this example we have $D=\{2,3\}$. In the grayed parts of the intervals, convergence did not necessarily occur, so the guess sequences might not necessarily agree. In the white parts of the intervals, convergence already occured, so we know the actual Grundy values.}
\end{figure}

If in an interval $I_m$ convergence does not occur within $c$ steps, then the algorithm will detect this and return an error. Under no condition will the algorithm return an incorrect answer.

The running time and space of this algorithm is $O(c\log^{|D|} n)$.

\section{\boldmath Convergence in $\imark(\{1\},\{d_1,d_2\})$}

In this section we prove the following:

\begin{theorem}\label{thm_conv_1d1d2}
	For every $1<d_1<d_2$ there exists a constant $c=c(d_1,d_2)$ such that in the game $\imark(\{1\},\{d_1,d_2\})$ convergence occurs in at most $c$ steps.
\end{theorem}
Theorem \ref{thm_conv_1d1d2} covers the game $\imark(\{1\},\{2,3\})$ previously studied in \cite{FN}.

As it turns out, all cases of the form $2<d_1<d_2$, as well as all cases of the form $d_1=2$, $d_2\ge 4$, are relatively easy to prove, whereas the specific case $d_1=2$, $d_2=3$ requires a somewhat more complicated analysis.

\begin{definition}
A position $n$ is called a \emph{bottleneck} if $n$ is not divisible by $d_1$ nor by $d_2$.
\end{definition}

If $n$ is a bottleneck then its only option is $n-1$, and hence $\mathcal G(n)\in\{0,1\}$. Furthermore, if $n$ is a bottleneck then either $\mathcal G(n-1)=0$ or $\mathcal G(n)=0$.

\begin{definition}
A position $n$ is called \emph{$d_i$-conducive} for $i\in\{1,2\}$ if the position $d_in+1$ is a bottleneck.
\end{definition}

The significance of conducive positions is the following: Suppose position $n$ is $d_i$-conducive with $\mathcal G(n)=0$, and suppose we are computing a guess sequence $\mathcal G_{\overline\sigma}$ starting from some position $m<d_in$. Then this guess sequence will set $\mathcal G_{\overline\sigma}(d_in)>0$ (since $n$ is a division follower of $d_in$), and then it will converge to the correct value $\mathcal G_{\overline\sigma}(d_in+1)=\mathcal G(d_in+1)=0$. Hence, we observe the following:

\begin{observation}\label{obs_consecutive_bottl}
	Suppose positions $n$ and $n+1$ are both $d_i$-conducive, and suppose position $n+1$ is also a bottleneck. Then every guess sequence $\mathcal G_{\overline\sigma}$ starting from some position $m<d_in$ will have converged to the correct value by position $d_i(n+1)+1$.
\end{observation}

\begin{lemma}
	In every game of the form $\imark(\{1\},\{2,d_2\})$, $d_2\ge 4$, convergence occurs in at most $4d_2+3$ steps.
\end{lemma}

\begin{proof}
	Let $m=4d_2k$. Then positions $m+1$ and $m+3$ are both bottlenecks, so $n=m/2=2d_2k$ and $n+1=2d_2k+1$ are both $2$-conducive, and $n+1$ is also a bottleneck. Hence, by Observation \ref{obs_consecutive_bottl}, every guess sequence starting from position $m-1$ or earlier will converge by position $m+3$. Since positions of the form $4d_2k$ occur every $4d_2$ steps, the claim follows.
\end{proof}

\begin{lemma}
	In every game of the form $\imark(\{1\},\{d_1,d_2\})$, $2<d_1<d_2$, convergence occurs in at most $d_1^2d_2^2+d_2+1$ steps.
\end{lemma}

\begin{proof}
	Let $m=d_1^2d_2^2k$. Then position $m+1$ is a bottleneck. We now claim that at least one of the positions $m+d_1+1$, $m+d_2+1$ must also be a bottleneck. Indeed, suppose the opposite is the case. Then $d_2|(d_1+1)$ and $d_1|(d_2+1)$. The former property, together with $d_1<d_2$, implies $d_1=d_2-1$. But then $d_1$ divides both $d_2-1$ and $d_2+1$, implying that $d_1|2$. Contradiction.
	
	Choose $i\in\{1,2\}$ so that position $m+d_i+1$ is a bottleneck. Then $n=m/d_i$ and $n+1$ are both $d_i$-conducive, and $n+1$ is also a bottleneck. Hence, the claim follows from Observation \ref{obs_consecutive_bottl} as before.
\end{proof}

\begin{lemma}\label{lem_convS1D23}
	In the game $\imark(\{1\},\{2,3\})$ convergence occurs in at most $64$ steps.
\end{lemma}

\begin{proof}
	Note that in the game $\imark(\{1\},\{2,3\})$, the bottlenecks are the positions of the form $6k+1$ and $6k+5$. Therefore, the $2$-conducive positions are those of the form $3k$ and $3k+2$.
	
	\begin{lemma}\label{lem_five_consec}
		For every $n$ we have $\mathcal G(3(n+k))\neq 0$ for some $0\le k\le 4$. Meaning, there cannot be five consecutive positions divisible by $3$, all with Grundy value $0$.
	\end{lemma}
	
	\begin{proof}
		Since position $6k+1$ is a bottleneck, we have $\mathcal G(6k)=0$ or $\mathcal G(6k+1)=0$. Therefore (by considering $3\in D$), we have $\mathcal G(18k)\neq 0$ or $\mathcal G(18k+3)\neq 0$.
		
		Similarly, since position $6k+5$ is a bottleneck, we have $\mathcal G(6k+4)=0$ or $\mathcal G(6k+5)=0$. Therefore, we have $\mathcal G(18k+12)\neq 0$ or $\mathcal G(18+15)\neq 0$.
		
		Hence, the longest possible run of consecutive positions divisible by $3$ having Grundy value $0$ is of the form $18k+3, 18k+6, 18k+9, 18k+12$, which is of length $4$.
	\end{proof}
	
	\begin{lemma}
		For every $n$ we have $\mathcal G(3(n+k)+2) = 0$ or $\mathcal G(3(n+k)+3) = 0$ for some $0\le k\le 9$.
	\end{lemma}
	
	\begin{proof}
		Suppose for a contradiction that none of these Grundy values are $0$. Then for each $0\le k\le 9$ one of the division followers of $3(n+k)+3$ must have Grundy value $0$. If $n+k$ is even then the only divisor is $3$, while if $n+k$ is odd then both $2$ and $3$ are divisors. However, after dividing by $3$ we cannot have two adjacent positions with Grundy value $0$. Hence, whenever $n+k$ is even, the follower obtained by dividing by $3$ must have Grundy value $0$, whereas whenever $n+k$ is odd, the follower obtained by dividing by $2$ must have Grundy value $0$. Hence, after dividing by $2$ we have five consecutive positions that are divisible by $3$ and all have Grundy value $0$. This contradicts Lemma \ref{lem_five_consec}.
	\end{proof}
	Hence, within $64$ steps of any starting position we will encounter a bottleneck of the form $m=6k+1$ or $m=6k+5$ such that $(m-1)/2$ has Grundy value $0$.	This concludes the proof of Lemma \ref{lem_convS1D23}, and thus of Theorem \ref{thm_conv_1d1d2}.
\end{proof}

\section{Experimental results}

Using a computer program (see \texttt{measure\textunderscore convergence.cpp} in the ancillary files of the arXiv version) we checked how many steps it takes for convergence to occur in different $\imark$ games. See Table \ref{table_convergence}.

\begin{table}
	\centering
	\begin{tabular}{c|c|c}
		game&experimental&proven\\\hline
		$\imark(\{1\},\{2,3\})$&10&64\\
		$\imark(\{1\},\{2,4\})$&5&19\\
		$\imark(\{1\},\{2,5\})$&10&23\\
		$\imark(\{1\},\{3,4\})$&14&149\\
		$\imark(\{1\},\{3,5\})$&13&231\\
		$\imark(\{1\},\{4,5\})$&18&406\\
		$\imark(\{2\},\{2,3\})$&14&---\\
		$\imark(\{2\},\{2,4\})$&no convergence&---\\
		$\imark(\{2\},\{3,4\})$&18&---\\
		$\imark(\{3\},\{2,3\})$&24&---\\
		$\imark(\{1,2\},\{2,3\})$&21&---\\
		$\imark(\{1,3\},\{2,3\})$&no convergence&---\\
		$\imark(\{1,4\},\{2,3\})$&$6551^*$&---\\
		$\imark(\{2,3\},\{2,3\})$&$5816^*$&---\\
		$\imark(\{1\},\{2,3,5\})$&12&---\\
	\end{tabular}
	\caption{\label{table_convergence}Experimental and proven upper bounds for number of steps to convergence, for different $\imark$ games. The experimental upper bounds are for starting positions up to $10^6$. Proven upper bounds are only given for games of the form $\imark(\{1\},\{d_1,d_2\})$. Experimental bounds marked~${}^*$ seem unreliable, since the measurement up to starting position $n$ increases steadily as $n$ itself increases.}
\end{table}

We also implemented the convergence-based algorithm for games of the form $\imark(\{1\},\allowbreak\{d_1,d_2\})$ (see \texttt{recursive\textunderscore alg.cpp} in the ancillary files). For the games shown in Table \ref{table_Gvalues}, we checked that the convergence-based algorithm agrees with the naive algorithm up to $n=10^6$, and we also computed $\mathcal G(n)$ for the large values of $n$ shown in the table.

\begin{table}
	\centering
	\begin{tabular}{c|c}
		game & $\mathcal G(n), \ldots, \mathcal G(n+30)$ for $n=10^{18}$\\\hline
		$\imark(\{1\},\{2,3\})$& 2 0 1 0 1 2 0 1 0 1 2 1 0 1 0 1 2 0 2 0 1 0 2 0 1 0 3 0 1 2 0\\
		$\imark(\{1\},\{2,4\})$&3 0 1 0 2 0 1 0 3 0 1 0 2 0 1 0 1 0 1 0 2 0 1 0 3 0 1 0 2 0 1\\
		$\imark(\{1\},\{2,5\})$&1 0 1 0 2 1 2 0 1 0 2 0 1 0 1 2 0 1 2 0 1 0 1 0 2 0 1 0 2 0 1\\
		$\imark(\{1\},\{3,4\})$&2 0 1 0 1 0 1 0 2 0 1 0 1 0 1 0 1 0 1 0 3 0 1 2 1 0 1 0 2 1 0\\
		$\imark(\{1\},\{3,5\})$&1 0 1 0 1 0 1 0 1 0 1 0 1 0 1 0 1 0 1 0 1 0 1 0 1 0 1 0 1 0 1\\
		$\imark(\{1\},\{4,5\})$&2 0 1 0 1 2 0 1 0 1 0 1 2 0 1 2 0 1 0 1 3 0 1 0 2 0 1 0 1 0 1
	\end{tabular}
	\caption{\label{table_Gvalues}Grundy values of some large $\imark$ positions, computed using the convergence-based algorithm. These computations took only a fraction of a second.}
\end{table}

\section{\boldmath $\imark$ games with no convergence}\label{sec_no_conv}

Not all $\imark$ instances experience convergence.

If $D=\emptyset$ there is no convergence, because if the guess seed is a translation of the real Grundy sequence, then the guess sequence will stay that way forever.

Hence, if all elements in $S\cup D$ are divisible by some number $k>1$, then there is no convergence, since from positions not divisible by $k$ one cannot reach any position divisible by $k$, and positions not divisible by $k$ have no divisor options, so they behave as if $D=\emptyset$.

Here is a more interesting example of an $\imark$ game with no convergence:

\begin{lemma}
	The game $\imark(\{1, 3\}, \{2, 3\})$ has no convergence.
\end{lemma}

\begin{proof}
	For convenience, let $N$ denote a position with nonzero Grundy value (or an $N$-position), let $T$ denote a position with Grundy value of $2$ or $3$, let $L$ denote a position with Grundy value different from $1$, and let $?$ denote a position with arbitrary Grundy value.
	
	First, it can be easily verified that the actual Grundy sequence, starting from position $6$, is of the form $(N, 0, N, ?, N, 0)^*$, where ${}^*$ denotes repetition forever.
	
	\begin{figure}
		\centering
		\includegraphics[width=16cm]{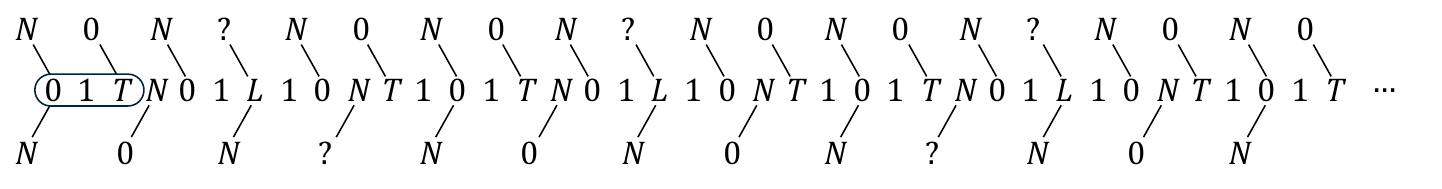}
		\caption{\label{fig_S13D23}Non-convergence in $\imark(\{1,3\},\{2,3\})$. The upper and lower rows contain the correct Grundy values, and the central row contains a guess sequence that differs forever from the correct Grundy sequence. Here $N$ denotes a nonzero number, $T$ denotes $2$ or $3$, $L$ denotes a number different from $1$, and $?$ denotes an arbitrary number.}
	\end{figure}
	
	Next, it can be easily verified that a seed of the form $(0, 1, T)$ starting at a position of the form $12k$ will form a guess sequence of the form $(0, 1, T, N, 0, 1, L, 1, 0, N, T, 1)^*$. This guess sequence keeps disagreeing forever with the actual Grundy sequence. See Figure \ref{fig_S13D23}.
\end{proof}

\section{Concluding remarks}

Games of $\imark$ with two or more divisors ($|D|\ge 2$) present unique challenges. Sopena \cite{SOPENA201690} mentioned them as a subject of possible future research. One such game was partially studied in \cite{FN}. In this paper we presented a convergence-based approach that works for some of these games. While our approeach does not identify precise patterns in the Grundy sequences (as was done in \cite{FN,SOPENA201690} for single-divisor games), it does yield polynomial-time algorithms for computing Grundy values.

In Theorem \ref{thm_conv_1d1d2}, we showed that convergence occurs in a certain subset of $\imark$ games with $|D|\ge 2$. Our experiments indicate that convergence might occur in a wider subset of such games, though not in all of them. It would be interesting to expand Theorem \ref{thm_conv_1d1d2} to cover a wider class of $\imark$ games.

In Section \ref{sec_no_conv} we presented some examples of $\imark$ games with no convergence. In some cases, the convergence approach might perhaps still be salvaged, by making a modification tailored specifically to each game; namely, by not trying guess seeds that are known to be wrong. For example, for the case of $\imark(\{1,3\},\{2,3\})$, if we start from positions of the form $6k$ and only use guess seeds of the form $(N,0,N)$, then convergence always seems to occur within $26$ steps.

It would also be interesting to check whether the convergence approach presented in this paper can be applied to other games.

\bibliographystyle{plainurl}
\bibliography{imark_convergence}

\end{document}